\documentclass[12pt]{article}
\usepackage{amsmath,amssymb,amsthm,fullpage}

\title{The Convergence of the Catalan Number Generating Function}

\author{Mark B. Villarino\\
Escuela de Matem\'atica, Universidad de Costa Rica,\\
11501 San Jos\'e, Costa Rica}

\date{}

\theoremstyle{plain}
\newtheorem{thm}{Theorem}             

\makeatletter
\def\section{\@startsection{section}{1}{\z@}{-3.5ex plus -1ex minus
			  -.2ex}{2.3ex plus .2ex}{\large\bf}}
\def\subsection{\@startsection{subsection}{2}{\z@}{-3.25ex plus -1ex
			  minus -.2ex}{1.5ex plus .2ex}{\normalsize\bf}}
\makeatother

\renewcommand{\geq}{\geqslant}  
\renewcommand{\leq}{\leqslant}  

\newcommand{\word}[1]{\quad\mbox{#1}\quad} 

\newcommand{\hideqed}{\renewcommand{\qed}{}} 

\begin{document}

\maketitle

\section{Introduction} 

In a letter dated September 4, 1751, \textsc{Leonhard Euler} proposed
the following problem to his friend, \textsc{Christian Goldbach} 
\cite[Appendix~B, p.~178]{Stanley}:
\emph{In how many ways, $T_n$, can a convex polygon of $n$~sides be
partitioned into triangles by diagonals which do not intersect within
the polygon?}
The statement of the problem is quite easy to understand, yet its
general solution leads to extraordinary difficulties. Euler worked out
the first few cases $T_3 = 1$, $T_4 = 2$, $T_5 = 5$ by actually
drawing the triangulations and computed (without diagrams) $T_6 = 14$,
$T_7 = 42$, $T_8 = 132$, $T_9 = 429$, $T_{10} = 1430$ and then
conjectured the formula
\begin{equation}
\label{Tn} 
T_n = \frac{2\cdot 6\cdot 10\cdot 14\cdot 18\cdot 22\cdots (4n - 10)}
{2\cdot 3\cdot 4\cdot 5\cdot 6\cdot 7\cdots (n - 1)} \,.
\end{equation}
At the end of his letter Euler guessed the generating function
$$
T(x) := 1 + 2x + 5x^2 + 14x^3 + 42x^4 + 132x^5 +\cdots 
= \frac{1 - 2x - \sqrt{1 - 4x}}{2x^2} \,,
$$
and added:
\begin{quote}
``However, the induction that I employed was pretty tedious, and I do
not doubt that this result can be reached much more easily.''
\end{quote}

Goldbach answered a month later, observing that the generating
function satisfies the quadratic equation
\begin{equation}
\label{Q} 
1 + xT = T^{1/2},
\end{equation}
which is equivalent to infinitely many equations in the coefficients
and suggested that they may lead to a direct proof of Euler's
formula~\eqref{Tn}.

It will be convenient to change the notation. We let
$$
C_{n-2} := T_n.
$$
Thus, $C_n$ is the number of triangulations of a convex $(n + 2)$-gon,
and since a little algebra shows that
$$
T_n = \frac{1}{n - 1} \binom{2n - 4}{n - 2},
$$
where $\binom{n}{m} = \frac{n!}{m!(n-m)!}$ is the binomial
coefficient, we obtain the famous formula:
\begin{equation}
\label{cat} 
C_n = \frac{1}{n + 1} \binom{2n}{n}
\end{equation}
for what today is called the $n^\mathrm{th}$ \emph{Catalan number}. We
note that $C_0 = 1$. The generating function is now
\begin{equation}
\label{Cx} 
C(x) := C_0 + C_1x + C_2x^2 +\cdots+ C_nx^n +\cdots
\end{equation}
and is given by the formula
$$
C(x) = \frac{1 - \sqrt{1 - 4x}}{2x} \,.
$$

Some time later Euler suggested the problem to \textsc{Johann Andreas
von Segner} who, in 1758, published the recursion 
formula~\cite{Segner}:
\begin{equation}
\label{Rec} 
C_{n+1} 
= C_0 C_n + C_1 C_{n-1} + C_2 C_{n-2} +\cdots+ C_{n-1} C_1 + C_n C_0,
\end{equation}
together with a combinatorial proof. However, he apparently was
unaware of Euler's explicit product formula~\eqref{Tn} since he never
mentions it. Instead he uses~\eqref{Rec} to directly compute $C_n$ for
$n = 1,2,\dots,18$.

Euler's letter is \emph{the first known publication of the Catalan
numbers}, and with Segner's recursion formula all the tools were
available for the modern development.

Today, of course, the standard treatment of $C_n$ is to use Segner's
recursion formula~\eqref{Rec} to deduce Goldbach's quadratic
equation~\eqref{Q} which leads to the generating function \eqref{Cx}
and finally to the explicit formula~\eqref{cat}, but eighty years
would pass before \textsc{Binet} would publish such a
proof~\cite{Binet}.

The Catalan numbers continue to fascinate mathematicians, and 
\textsc{Richard Stanley} recently published a book devoted exclusively
to them~\cite{Stanley} in which he presents 214~(!) combinatorial
interpretations of~$C_n$.

In 1967, 
\textsc{Marshall Hall} published a text on combinatorics~\cite{Hall}
and on page~28 we find the following comment (the notation has been
slightly altered):
\begin{quote}
``We observe that an attempt to prove the convergence of~\eqref{Cx} on
the basis of \eqref{Rec} alone is exceedingly difficult.''
\end{quote}
Hall offers no suggestions towards such a proof. Moreover, a search of
the voluminous literature on Catalan numbers has failed to find such a
proof. Therefore we offer a proof in this paper.

\section{Heuristics} 

In order to prove the convergence of \eqref{Cx} we have to show that
the power series $C(x)$ has \emph{a positive radius of convergence,
$R$}. By the \textsc{Cauchy--Hadamard} theorem from elementary
analysis we have to show that
\begin{equation}
\label{r1} 
\frac{1}{R} = \limsup_{n\to\infty} |C_n|^{1/n} > 0.
\end{equation}

Now, if we could show that there exists a positive constant $M$ such
that the inequality
\begin{equation}
\label{c1} 
C_n \leq M^n
\end{equation}
holds for all  $n \geq 0$, then we could conclude that
$R \geq \frac{1}{M} > 0$. But Segner's recursion formula~\eqref{Rec}
would give us
\begin{align*}
C_{n+1} &\leq M^0\cdot M^n + M^1\cdot M^{n-1} + M^2\cdot M^{n-2}
+\cdots+ M^{n-1}\cdot M + M^n\cdot M^0
\\
&= M^n (n + 1),
\end{align*}
which is a second upper bound. If we could show that it is smaller
than $M^{n+1}$ then we could conclude by induction that \eqref{c1}
holds for all~$n$. Unfortunately,
$$
M^n(n + 1) \leq M^{n+1} \implies n + 1 \leq M
$$
which is plainly false for all sufficiently large~$n$. So, the second
upper bound is, in fact, \emph{larger} than the first one and this
shows us that the induction step does not work for a bound of the
form~\eqref{c1}.
Thus, one has to alter~\eqref{c1} so as to somehow ``cancel'' the
factor $n + 1$ which multiplies $M^n$. This suggests that we should
try an inequality of the form 
$$
C_n \leq n^{-r} M^n
$$
for some positive integer~$r$, yet to be determined, with $n \geq 1$
and $C_0 = 1$, where the factor $n^{-r}$ effects the cancellation.

If we try $r = 1$, i.e., if we assume that the inequality
\begin{equation}
\label{c3} 
C_n \leq n^{-1} M^n
\end{equation}
holds for all $n \geq 1$, the recursion formula~\eqref{Rec} gives us
\begin{align*}
C_{n+1} &\leq 1\cdot \frac{M^n}{n}
+ \frac{M^1}{1} \cdot \frac{M^{n-1}}{(n-1)}
+ \frac{M^2}{2} \cdot \frac{M^{n-2}}{(n-2)} 
+\cdots+ \frac{M^{n-1}}{(n-1)} \cdot \frac{M}{1}
+ \frac{M^n}{n} \cdot 1
\\
&= M^n \biggl( \frac{2}{n}
+ \sum_{k=1}^{n-1} \frac{1}{k(n - k)} \biggr)
\end{align*}
and by~\eqref{c3}, we want
$$
M^n \biggl( \frac{2}{n}
+ \sum_{k=1}^{n-1} \frac{1}{k(n - k)} \biggr).
\leq \frac{M^{n+1}}{(n + 1)}
$$

But, if $H_k := 1 + \frac{1}{2} +\cdots+ \frac{1}{k}$ denotes the
$k^\mathrm{th}$ harmonic number,
then the identity $\frac{1}{k(n-k)} 
\equiv \frac{1}{n} \bigl( \frac{1}{k} + \frac{1}{n-k} \bigr)$ shows us
that 
$$
\frac{2}{n} + \sum_{k=1}^{n-1} \frac{1}{k(n - k)}
= \frac{1}{n}(2 + 2H_{n-1}) < \frac{M}{n+1} \,,
$$
or $H_{n-1} < \frac{n}{n+1} - 2$, which is false for all~$n$.

However, if we try $r = 2$, i.e., if we assume that the inequality
\begin{equation}
\label{c4} 
C_n \leq n^{-2} M^n
\end{equation}
holds for all $n \geq 1$, the recursion formula~\eqref{Rec} gives us
\begin{align*}
C_{n+1} &\leq 1 \cdot \frac{M^n}{n^2}
+ \frac{M^1}{1^2} \cdot \frac{M^{n-1}}{(n-1)^{2}}
+ \frac{M^2}{2^2} \cdot \frac{M^{n-2}}{(n-2)^2} 
+\cdots+ \frac{M^{n-1}}{(n-1)^2} \cdot \frac{M}{1^2}
+ \frac{M^n}{n^2} \cdot 1
\\
&= M^n \biggl( \frac{2}{n^2}
+ \sum_{k=1}^{n-1} \frac{1}{\{k(n - k)\}^2} \biggr)
\end{align*}
and by~\eqref{c4}, we want
$$
M^n \biggl( \frac{2}{n^2}
+ \sum_{k=1}^{n-1} \frac{1}{\{k(n - k)\}^2} \biggr)
\leq \frac{M^{n+1}}{(n + 1)^2} \,,
$$
or
\begin{equation}
\label{c5} 
\frac{2}{n^2} + \sum_{k=1}^{n-1} \frac{1}{\{k(n - k)\}^2}
\leq \frac{M}{(n + 1)^2} \,.
\end{equation}

We will show that \emph{if $M = 6$ then the inequality~\eqref{c5}
holds for all $n \geq 37$.} By direct numerical computation it can be
verified that \eqref{c3} holds for $n = 1,\dots,36$ and that therefore
the inequality~\eqref{c3} holds for all $n \geq 1$. Therefore
$R \geq \frac{1}{6} > 0$.

We add that \textsc{Stirling's} formula and~\eqref{cat} show that 
$$
C_n \sim \frac{4^n}{\sqrt{\pi}\, n^{3/2}} \,,
$$
so that the true value of $R$ is $R = \frac{1}{4}$. However, this
presupposes the knowledge of the explicit formula for $C_n$, whereas
our analysis makes no such assumption.

\section{The Proof} 

\begin{thm} 
The following limit relation is valid:
\begin{equation}
\label{lema} 
\lim_{n\to\infty} \biggl\{ \biggl(
\frac{2}{n^2} + \sum_{k=1}^{n-1} \frac{1}{\{k(n - k)\}^2} \biggr)
\biggm/ \frac{1}{(n + 1)^2} \biggr\}
= 2 + \frac{\pi^2}{3}
\end{equation}
and, moreover, if the integer $n \geq 4$ then the quotient on the left
decreases monotonically to its limit. This may also be written as the
following asymptotic equality:
$$
\frac{2}{n^2} + \sum_{k=1}^{n-1} \frac{1}{\{k(n - k)\}^2}
\sim \frac{2 + \frac{\pi^2}{3}}{(n + 1)^2} \,.
$$
Moreover, if $n \geq 37$, then 
\begin{equation}
\label{lema1} 
\frac{2 + \frac{\pi^2}{3}}{(n + 1)^2}
< \frac{2}{n^2} + \sum_{k=1}^{n-1} \frac{1}{\{k(n - k)\}^2}
< \frac{6}{(n + 1)^2} \,.
\end{equation}
\end{thm}

\begin{proof}
We suppose that the integer $n$ satisfies $n\geq 4$.

In what follows, we use the 
well-known elementary inequality $H_n < 1 + \ln n$, and
Euler's famous formula 
$$
\frac{1}{1^2} + \frac{1}{2^2} + \frac{1}{3^2} +\cdots 
= \frac{\pi^2}{6} \,.
$$
We begin with the left-hand side of \eqref{lema}: 
\begin{align}
\frac{2}{n^2} + \sum_{k=1}^{n-1} \frac{1}{\{k(n - k)\}^2}
&= \frac{2}{n^2} + \sum_{k=1}^{n-1} \biggl\{ \frac{1}{n} \biggl(
\frac{1}{k} + \frac{1}{n - k} \biggr) \biggr\}^2
\nonumber \\   
&= \frac{2}{n^2} + \frac{1}{n^2} \sum_{k=1}^{n-1} \biggl\{
\frac{1}{k^2} + \frac{2}{k(n-k)} + \frac{1}{(n-k)^2} \biggr\}
\nonumber \\
&= \frac{2}{n^2} 
+ \frac{1}{n^2} \sum_{k=1}^{n-1} \biggl\{ \frac{2}{k^2}
+ \frac{2}{n} \biggl( \frac{1}{k} + \frac{1}{n - k} \biggr) \biggr\}
\nonumber \\
&= \frac{2}{n^2} + \frac{2}{n^2} \sum_{k=1}^{n-1} \frac{1}{k^2}
+ \frac{4}{n^3}\, H_{n-1}
\nonumber \\
&= \frac{2}{n^2} \biggl\{ 1 + \frac{\pi^2}{6}
- \sum_{k=n}^\infty \frac{1}{k^2} + \frac{2}{n}\, H_{n-1} \biggr\}.
\label{trunc} 
\end{align}
    
Let
$$
g(n) := \biggl( \frac{2}{n^2}
+ \sum_{k=1}^{n-1} \frac{1}{\{k(n - k)\}^2} \biggr)
\biggm/ \frac{1}{(n + 1)^2} \,.
$$
The computation \eqref{trunc} shows that 
$$
g(n) = 2\biggl( 1 + \frac{1}{n} \biggr)^2 \biggl\{ 1 + \frac{\pi^2}{6}
- \sum_{k=n}^\infty \frac{1}{k^2} + \frac{2}{n}\, H_{n-1} \biggr\}.
$$
Therefore, $\lim_{n\to\infty} g(n) = 2 + \pi^2/3$,
establishing~\eqref{lema}.

To prove monotonicity, we have to show that
\begin{equation}
\label{gn} 
g(n) > g(n + 1).
\end{equation}
But, since clearly
$$
2\biggl( 1 + \frac{1}{n} \biggr)^2
> 2\biggl( 1 + \frac{1}{n + 1} \biggr)^2,
$$  
it suffices to prove that
$$
1 + \frac{\pi^2}{6} - \sum_{k=n}^\infty \frac{1}{k^2}
+ \frac{2}{n}\,H_{n-1}
> 1 + \frac{\pi^2}{6} - \sum_{k=n+1}^\infty \frac{1}{k^2}
+ \frac{2}{n + 1} \,H_n \,;
$$
which, after a little algebra, reduces to proving
$$ 
H_{n-1} - \frac{n}{n + 1}\,H_n > \frac{1}{2n} \,,
$$
or
$$
\frac{1}{2} + \frac{1}{3} +\cdots+ \frac{1}{n-1}
> \frac{1}{2} + \frac{1}{2n} \,,
$$
which holds if and only if $n \geq 4$. This completes the proof of
monotonicity.

To prove \eqref{lema1}, we compute directly
$$
g(36) = 6.0150\dots  \word{and}  g(37) = 5.9979\dots
$$
and by the monotonicity of $g(n)$ we conclude that 
$$
2 + \frac{\pi^2}{3} < g(n) < 6  \word{for}  n \geq 37.
\eqno \qed 
$$
\hideqed   
\end{proof}

\begin{thm} 
The Catalan number generating function
$$
C(x) := C_0 + C_1 x + C_2 x^2 +\cdots+ C_n x^n + \cdots
$$
has a positive radius of convergence at least equal
to~$\frac{1}{6}\,$.
\end{thm}

\begin{proof}
The previous theorem shows that \eqref{c3} holds for $n \geq 37$, and
by our earlier remarks, for $n \geq 1$. Therefore by the
Cauchy--Hada\-mard theorem \eqref{r1}, we conclude that
$R \geq \frac{1}{6} > 0$.
\end{proof}

\section{Remarks} 

Our method of proof is applicable to any integer $r = 2,3,\dots$. That
is, for any integer $r \geq 2$ there is a constant $M_r$ such that the
inequality
\begin{equation}
\label{cr} 
C_n \leq \frac{M_r^n}{n^r}
\end{equation}
holds for all sufficiently large~$n$. Moreover, the validity
of~\eqref{cr} for $r = 3$ shows that the inequality \eqref{cr} holds
for any~$r$ such that $2 \leq r \leq 3$. Therefore, the method is
applicable to any real number~$r$ which satisfies $r \geq 2$.

On the other hand, although the inequality~\eqref{cr} holds for every
integer $r \geq 2$, our proof shows that it does \emph{not} hold for
$r = 1$. So, this suggests that there is a first value $r = r_0$
between $1$ and~$2$ for which the inequality~\eqref{cr} is valid. The
methodology of the proof and Stirling's formula,
$$
C_n \sim \frac{4^n}{\sqrt{\pi}\,n^{3/2}}
$$
show that, in fact,
$$
r_0 = \frac{3}{2} \,,
$$
i.e., if $r \geq \frac{3}{2}$ then there exists a constant $M_r > 0$
such that the inequality~\eqref{cr} holds for sufficiently large~$n$
and that therefore the radius of convergence is positive. On the other
hand, if $r < \frac{3}{2}$, the inequality~\eqref{cr} does not hold.

It would be interesting to have a direct proof of this property of
$r = r_0$ without appealing to the explicit formula for the Catalan
numbers.

\subsection*{Acknowledgment}

Support from the Vicerrector\'ia de Investigaci\'on of the University
of Costa Rica is gratefully acknowledged. I thank Joseph C. V\'arilly
for a helpful remark.

\end{document}